\documentclass[11pt,reqno]{amsart}
\usepackage[usenames]{color}
\usepackage{amsmath, amssymb, latexsym, multicol, amsthm, color, enumitem,comment}
\usepackage[all]{xy}
\newtheorem{theorem}{Theorem}
\newtheorem{lemma}[theorem]{Lemma}

\newtheorem{proposition}[theorem]{Proposition}
\newtheorem{conjecture}[theorem]{Conjecture}

\theoremstyle{remark}

\newcommand{\Z}{\mathbb{Z}}
\newcommand{\Q}{\mathbb{Q}}
\newcommand{\N}{\mathbb{N}}

\newcommand{\R}{\mathbb{R}}
\newcommand{\C}{\mathbb{C}}

\renewcommand{\H}{\mathbb{H}}

\newcommand{\gen}{\operatorname{gen}}

\newcommand{\spn}{\operatorname{spn}}
\newcommand{\lcm}{\operatorname{lcm}}
\newcommand{\id}{\operatorname{id}}
\newcommand{\cls}{\text{cls}}
\newcommand{\lan}{\langle}
\newcommand{\ran}{\rangle}

\newcommand{\SL}{\operatorname{SL}}

\numberwithin{equation}{section}
\numberwithin{theorem}{section}

\title[almost universal weighted ternary sums of polygonal numbers]{almost universal weighted ternary sums \\ of polygonal numbers}

\author{Siu Hang Man}
\address{Siu Hang Man, Deparmtent of Mathematics, University of Hong Kong, Pokfulam, Hong Kong.}
\email{der.gordox@gmail.com}

\author{Archie Mehta}
\address{Archie Mehta, Department of Mathematics, Indian Institute of Technology, Roorkee.}
\email{archiemehta.iitr@gmail.com}

\thanks{The research of the first author was supported by Professor Rosie Young Fund and the research of the second author was supported by a research assistancship on the grant of Dr. Ben Kane, grant project number 27300314 of the Research Grants Council of Hong Kong SAR, China. }

\begin{document}

\begin{abstract}
For a natural number $m$, generalized $m$-gonal numbers are defined by the formula $p_m(x)=\frac{(m-2)x^2-(m-4)x}{2}$ with $x\in \mathbb Z$. In this paper, we determine a criterion on $a,b,c,m$ for which the weighted ternary sum $P_{a,b,c,m}:=ap_m(x)+bp_m(y)+cp_m(z)$ is almost universal. We also prove for some $a,b,c,m$ that the form $P_{a,b,c,m}$ is not almost universal, while it represents all possible congruence classes.
\end{abstract}

\date{\today}
\maketitle

\section{Introduction} 

For $m\in\N$, the $x$-th generalized $m$-gonal number is defined to be $p_m(x)=\frac{(m-2)x^2-(m-4)x}{2}$, where $x\in\Z$.  These sets of numbers have aroused the interest of many mathematicians in the past centuries. In 1813, Cauchy proved that every natural number is a sum of  at most $m$ many $m$-gonal numbers, a conjecture raised by Fermat in 1638 \cite{Cauchy}. Nevertheless, for large enough $n\in\N$, it is likely that one can represent $n$ with many fewer $m$-gonal numbers. Indeed, Haensch and Kane \cite{HaenschKane} proved that for $m\not\equiv 2\pmod{3}$ and $4\nmid m$, every large enough $n\in\N$ can be written as the sum of 3 $m$-gonal numbers.

For fixed values of $m$ (e.g. $m=4$ for square numbers), it is also natural to consider representations of integers by weighted sums. For example, for $m=4$, one considers representations of $n\in\N_0$ as ternary weighted sums of squares, i.e., 
\begin{align}\label{eqn:sumssquares}
n=ax^2+by^2+cz^2,
\end{align}
where $a,b,c\in\N$ are considered to be fixed.  The first consideration is a local one, since integers not represented modulo a given modulus clearly cannot be represented globally. For example, the entire congruence class $n=8k+7$ cannot be represented as the sum of $3$ squares because this identity is not soluble in $\Z/ 8\Z$. 
Unfortunately, this is part of a much more general phenomenon. No matter how $a,b,c$ are chosen, there is always a {\it local obstruction} for \eqref{eqn:sumssquares};
that is, there exists a congruence class that is not represented at all. Nevertheless, this suggests another possible direction of generalisation to the problem: for a fixed value of $m$, what weights $a,b,c$ 
can one choose so that
\begin{align}\label{eqn:polyweighted}
P_{a,b,c,m}(x,y,z):=ap_m(x)+bp_m(y)+cp_m(z)
\end{align}
is {\it almost universal}, (i.e., represents all sufficiently large $n\in\N$). In this direction, Kane and Sun \cite{KaneSun} 
obtained a near-classification of almost universal weighted sums of triangular numbers (corresponding to $m=3$ and more generally weighted mixed ternary sum of triangular and square numbers;
this classification was later completed by Chan--Oh \cite{ChanOh} and Chan--Haensch \cite{ChanHaensch}.

As demonstrated above, this traditional old problem has two main directions of generalisation, one on the value of $m$, and the other on the weighting $a,b,c$. In this paper, we combine these two directions, allowing $a,b,c$ and $m$ all to vary and studying whether the form $P_{a,b,c,m}$ in \eqref{eqn:polyweighted} is almost universal or not.  We begin with a necessary condition about local obstructions; namely, for a form to be almost universal, it must represent all possible congruence classes. However, the absence of local obstruction does not guarantee that the form is almost universal. To better investigate this issue, we define
\[\mathcal{S}_{a,b,c,m}=\big\{n\in\N | \not\exists (x,y,z)\in\Z^3 \text{ with } P_{a,b,c,m}(x,y,z)=n\big\},\]
\[\mathcal{S}_{a,b,c,m}^t = \big\{n\in\mathcal{S}_{a,b,c,m} | \exists r\in\Z \text{ such that } 8(m-2)n+(a+b+c)(m-4)^2 = tr^2\big\},\]
and for $T\subseteq\N$,
\[
\mathcal{S}_{a,b,c,m}^T = \bigcup\limits_{t\in T} \mathcal{S}_{a,b,c,m}^t.
\]
Roughly speaking, one may use the analytic theory of quadratic forms to show that if the form $P_{a,b,c,m}$ has no local obstructions and every integer is locally ``primitively'' represented, then there exists a finite subset $T$ of $\N$ for which the set $\mathcal{S}_{a,b,c,m}\backslash\mathcal{S}_{a,b,c,m}^T$ is finite. Clearly, if $T=\emptyset$, then the form $P_{a,b,c,m}$ is almost universal. As a result, we obtain the following theorem on almost universality:
\begin{theorem}\label{aucrit}
Suppose $P_{a,b,c,m}$ is a form without local obstructions (i.e. every possible congruence class is represented by the form). Let $\ell_n :=8(m-2)n+(a+b+c)(m-4)^2$, $u = \gcd_{n\geq 0} \ell_n$, $s$ be the square part of $u$, $\ell'_n = \ell_n/s$, and $T$ be the set of positive integral divisors of $\gcd(4(m-2),\ell'_n)$. If the congruence equation
\begin{align}
tr^2\equiv(a+b+c)(m-4)^2 \pmod{8(m-2)}
\end{align}
is not solvable for every $t\in T$, then $P_{a,b,c,m}$ is almost universal.
\end{theorem}

It is natural to investigate whether there exist forms with no local obstructions (and for which the ternary weighted sum of squares primitively represents the congruence class of interest) which are not almost universal.  In order to study this, we rely on the approach introduced by Haensch and Kane \cite{HaenschKane}, utilising analogues and generalisations of the Siegel--Weil formula. Let $L_0$ be a lattice in a positive definite space (for our purpose, we assume $\dim(L_0)=3$), and $\cls/\gen(L_0)$ be the set of class representatives in the genus of $L_0$. The Siegel--Weil formula states that the weighted sum of theta series
\begin{align}
\mathcal{E}_{\gen(L_0)} :=\left(\sum\limits_{L\in\cls/\gen(L_0)}\omega_L^{-1}\right)^{-1}\sum\limits_{L\in\cls/\gen(L_0)}\frac{\Theta_L}{\omega_L}
\end{align}
is an Eisenstein series. Here  $\Theta_L$ denotes the theta function associated to $L$, defined by
\[
\Theta_L(\tau)=\sum\limits_{n\geq 0} r_L(n)q^n,
\]
where $r_L(n)$ is the number of solutions $x\in\Z^3$ to the equation $x^TLx=n$; and $\omega_L$ is the number of automorphs of the lattice $L$ (i.e. the number of linear transformations that map $L$ to itself).

An analogue that sums over representatives in the spinor genus $\spn(L_0)$ was proven by Schulze-Pillot \cite{SPillot1,SPillot2}:
\begin{align}
\left(\sum\limits_{L\in\cls/\spn(L_0)}\omega_L^{-1}\right)^{-1}\sum\limits_{L\in\cls/\spn(L_0)}\frac{\Theta_L}{\omega_L}= \mathcal{E}_{\gen(L_0)}+\mathcal{U}_{\spn(L_0)},
\end{align}
where $\mathcal{U}_{\spn(L_0)}$ is a linear combination of unary theta functions (see (\ref{utfdef})). 

A generalisation of the Siegel-Weil formula from a lattice $L$ to an arbitrary lattice coset $L+\nu$ (with $\nu\in \Q L$) was proven by Shimura \cite{ShimuraLattCo}; namely, the weighted sum of theta series
\begin{align}\label{ShimuraSW}
\mathcal{E}_{\gen(L+\nu)} :=\left(\sum\limits_{M+\mu\in\cls/\gen(L+\nu)}\omega_{M+\mu}^{-1}\right)^{-1} \sum\limits_{M+\mu\in\cls/\gen(L+\nu)}\frac{\Theta_{M+\mu}}{\omega_{M+\mu}}
\end{align}
is also an Eisenstein series. Here $\gen(L+\nu)$ (see \eqref{gencosetdef}), $\Theta_{M+\mu}$, and $\omega_{M+\mu}$ are defined analogously to the their lattice counterparts.

Haensch and Kane \cite{HaenschKane} conjectured the spinor genus analogue of the generalised Siegel--Weil formula by Shimura:
\begin{conjecture}\label{sswconj} The equation
\begin{align}\label{ssweqn}
\left(\sum\limits_{M+\mu\in\cls/\spn(L+\nu)}\omega_{M+\mu}^{-1}\right)^{-1} \sum\limits_{M+\mu\in\cls/\spn(L+\nu)}\frac{\Theta_{M+\mu}}{\omega_{M+\mu}} = \mathcal{E}_{\spn(L+\nu)}+\mathcal{U}_{\spn(L+\nu)}
\end{align}
holds, where $\mathcal{U}_{\spn(L+\nu)}$ is again a linear combination of unary theta functions.
\end{conjecture}

In \cite{HaenschKane}, this conjecture was verified for the lattice coset corresponding to the form $P_{1,1,1,14}$, and it was proven that $P_{1,1,1,14}$ is not almost universal even though it has no local obstructions. In this paper, we verify the conjecture for the lattice coset corresponding to the form $P_{1,1,3,17}$, a form with non-trivial weighting, and prove that $P_{1,1,3,17}$ is also not almost universal while having no local obstructions. The idea of proof is analogous to that in \cite{HaenschKane}. We consider the Fourier expansion of the equation \eqref{ssweqn}. Using the combinatorial interpretation of the Fourier coefficients of the theta series, one immediately concludes that these coefficients are always non-negative. Hence if the $\ell$th coefficient of $\mathcal{E}_{\spn(L+\nu)}+\mathcal{U}_{\spn(L+\nu)}$ vanishes, it must vanish for every $\Theta_{M+\mu}$ with $M+\mu\in \spn(L+\nu)$.  In particular, the $\ell$th coefficient of the theta series $\Theta_{L+\nu}$ counts the number of representations of $\ell$ by the form $P_{1,1,3,17}$, and we conclude that if infinitely many coefficients of $\mathcal{E}_{\spn(L+\nu)}+\mathcal{U}_{\spn(L+\nu)}$ are zero, then there are infinitely many integers $\ell$ which are not represented by the form $P_{1,1,3,17}$, i.e., $P_{1,1,3,17}$ is not almost universal.
\begin{proposition}\label{notauintro}
The form $P_{1,1,3,17}$ is not almost universal.
\end{proposition}

Meanwhile, the lattice coset $L+\nu$ corresponding to the form $P_{1,1,3,17}$ has sublattice cosets that correspond to the form $P_{1,1,3,30r+17}$, for some $r\in\N$. As a result, we can obtain a family of forms that are not almost universal. In particular, we prove the following theorem:

\begin{theorem}\label{notaugenintro}
For any non-negative integer $r\not\equiv 2\pmod{5}$, the form $P_{1,1,3,30r+17}$ is not almost universal.
\end{theorem}

The content of the paper is structured as follows. In Section 2, we give a brief introduction to the tools we use in later sections. In Section 3, we deal with the local conditions, that is, we prove Propositions \ref{localobsodd} and \ref{localobs2}. In Section 4, we prove Theorem \ref{aucrit}. Finally, in Section 5, we prove Proposition \ref{notauintro}, and its generalisation, Theorem \ref{notaugenintro}. We also explicitly construct infinite subsets of $\N$ that are not represented by these forms.

\section*{Acknowledgements}
We thank Ben Kane for his kind supervision and insightful comments.

\section{Preliminaries}

In this section, we give a brief introduction to mathematical tools that are utilised in the proofs.

\subsection{Modular forms}
Firstly, we fix some notation. Let $\H$ be the upper half complex plane, consisting of $\tau\in\C$ with Im$(\tau)>0$; and $\gamma = (\begin{smallmatrix} a&b\\ c&d \end{smallmatrix})$ be a matrix in $\SL_2(\Z)$ (the group of 2-by-2 matrices with integral entries and determinant 1). Then, we can define a group action of $\SL_2(\Z)$ on $\H$, by the transformation $\gamma\tau :=\frac{a\tau+b}{c\tau+d}$. Define also $j(\gamma,\tau):=c\tau+d$.

Also, for a subgroup $\Gamma\in\SL_2(\Z)$ and weight $k\in\R$, a multiplier system is a function $\chi: \Gamma\to\C$ such that for all $\gamma,M\in\Gamma$, we have
\[
\chi(M\gamma)j(M\gamma,\tau)^k=\chi(M)j(M,\gamma\tau)^k\chi(\gamma)j(\gamma,\tau)^k.
\]
We can then define the slash operator as:
\[
f|_{k,\chi}\gamma(\tau) = \chi(\gamma)^{-1}j(\gamma,\tau)^-k f(\gamma\tau).
\]

A modular form of weight $k\in\R$ with multiplier system $\chi$ on a subgroup $\Gamma\subseteq\SL_2(\Z)$ is a holomorphic function $f:\H\to\C$ satisfying these criteria:
\begin{enumerate}[label=(\alph*)]
\item For every $\gamma\in\Gamma$, we have
\begin{align*}
f|_{k,\chi}\gamma(\tau)=f(\tau);
\end{align*}

\item $f$ is bounded towards every cusp (i.e. images of $i\infty$ by $\SL_2(\Z)$). In other words, for every cusp $\varrho$, let $\gamma_\varrho$ be the element in $\SL_2(\Z)$ that maps $i\infty$ to $\varrho$, then the function $f_\varrho = f|_{k,\chi}\gamma_\varrho(\tau)$ is bounded when $\tau\to i\infty$.
\end{enumerate}

Furthermore, if $f$ vanishes at every cusp (i.e. $\lim\limits_{\tau\to i\infty} f_\varrho(\tau)=0$), then $f$ is called a cusp form.

\subsubsection{Half-integral weight modular forms}

The modular forms of particular interest in this paper are those with half-integral weight, that is, $k\in\frac{1}{2}+\N_0$, and on the subgroup 
\[ \Gamma=\Gamma_1(N):=\Big\{\begin{pmatrix}a&b\\c&d\end{pmatrix} \in\SL_2(\Z) \quad \Big| \quad c\equiv 0\pmod{N}, a\equiv d\equiv 1\pmod{N} \Big\},\]
where $N$ is an integer divisible by 4. The multiplier systems are given explicitly in \cite[Proposition 2.1]{Shimura1/2}, but it can be eliminated (i.e. collapse to the trivial multiplier system) by a substitution $\tau\to N\tau$, so we need not study them in detail.

In particular, for the modular forms concerned here, we have $T^N\in\Gamma$ for some $N\in\N$, where $T=(\begin{smallmatrix}1&1\\0&1\end{smallmatrix})$. Then we have $f(\tau+N)=f(\tau)$. This implies $f$ is periodic, and hence possesses a Fourier expansion
\begin{align*}
f(\tau)=\sum\limits_{n\geq 0} a_ne^\frac{2\pi in\tau}{N}.
\end{align*}
This Fourier expansion form is used throughout this paper. For simplicity, one usually adopts the notation $q:=2\pi i\tau$ to simplify the expression.

\subsection{Lattice Theory} An alternative approach to the problem is to make use of lattice theory. Here we follow the notation in \cite{HaenschKane, Omeara}. For a lattice $L$, we denote the underlying quadratic space by $V=\Q L$, and its localisation to prime $p$ by $L_p=L\otimes_\Z \Z_p$. Here, $L_p$ is a $\Z_p$-lattice corresponding to the space $V_p:=V\otimes_\Q\Q_p$.

Given a vector $\nu\in V$, we can define a lattice coset $L+\nu$, which simply consists of vectors in the form $x+\nu$, with $x\in L$. In order to present the Siegel-Weil formula, and its conjectured extension in \cite[Conjecture 1.3]{HaenschKane}, we need the notions of class, spinor genus, as well as genus of a lattice coset. We define the class of $L+\nu$ to be
\begin{align*}
\cls(L+\nu) = \text{ the orbit of } L+\nu \text{ under the action of }O(V), 
\end{align*}
the spinor genus of $L+\nu$ to be
\begin{align*}
\spn(L+\nu) = \text{ the orbit of } L+\nu \text{ under the action of }O(V)O'_\mathbb{A}(V), 
\end{align*}
and the genus of $L+\nu$ to be
\begin{align*}\label{gencosetdef}
\gen(L+\nu) = \text{ the orbit of } L+\nu \text{ under the action of }O_\mathbb{A}(V).
\end{align*}
Here $O_\mathbb{A}(V)$ denotes the adeles of $O(V)$. We denote by $\cls/\gen(L+\nu)$ (resp. $\cls/\spn(L+\nu)$) the set of class representatives of the spinor genus (resp. genus) of $L+\nu$.

In other words, if $M+\mu\in\gen(L+\nu)$, then for every prime $p$ there exists $X\in O(V_p)$ such that $(M+\mu)_p=X(L+\nu)_p$. Meanwhile, $O'\mathbb{A}(V)$ denotes the adeles of $O'(V)$, the kernel of the spinor norm map $\theta: O(V)\to \Q^\times / \Q^{\times^2}$. As introduced in \cite[p. 137]{Omeara}, let $X\in O(V)$ be a composition of symmetries $\tau_1\tau_2\cdots\tau_k$, the spinor norm of $X$ is defined to be
\begin{align*}
\theta(X):=\prod\limits_{i=1}^k Q(\tau_i)\Q^{\times^2},
\end{align*}
where $Q$ stands for the corresponding quadratic form in the quadratic space.

As the quadratic forms that are dealt with in this paper are diagonal, the lattice cosets $L+\nu$ and $L-\nu$ represent the same numbers. Hence, different from the notations set in \cite{Omeara, Xu}, we use $O(V)$ instead of $SO(V)$ as the definition of $\cls(L+\nu)$, reducing the number of classes by half.

\subsubsection{Hilbert Symbols and Hasse Algebra} Here we adopt the definition in \cite{Gerstein}. For prime $p$, and two nonzero elements $a,b\in\Q_p$, we define the Hilbert symbol $(a,b)_p$ as follows:
\[
(a,b)_p = 
\begin{cases}
1 &\quad\text{ if } ax^2+by^2 = 1 \text{ is solvable in }\Q_p,\\
-1 &\quad\text{ otherwise.}
\end{cases}
\]
For odd $p$, there is a simple formula for the computation of the Hilbert symbols. Write $a=p^\alpha u, b=p^\beta v$, where $u,v\in\mathfrak{u}$, the set of units in the ring of integers $\mathfrak{o}(\Q_p)$ of $\Q_p$. Then we have
\begin{align}\label{Hilbert}
(a,b)_p = (-1)^{\alpha\beta\frac{p-1}{2}}\Big(\frac{u}{p}\Big)^\beta\Big(\frac{v}{p}\Big)^\alpha,
\end{align}
where the brackets stand for the Legendre symbol. 

With Hilbert symbol, we can then define the Hasse symbol for a regular quadratic space $V\equiv\langle a_1, \cdots, a_n\rangle$ to be
\begin{align*}
S_pV = \prod_{i<j} (a_i, a_j)_p.
\end{align*}
Then, for a ternary regular quadratic $\Q_p$-space $V\cong\langle a,b,c\rangle$, it is known that \cite[Proposition 4.21]{Gerstein} 
\begin{align}\label{isocrit}
V \text{ is isotropic } \Leftrightarrow S_pV=(-1,-dV)_p,
\end{align}
where $dV$ stands for the discriminant of $V$.

\section{Local Conditions}

In this section, we find the criteria on $a,b,c,m$ such that every congruence class is represented. Consider the equation
\begin{align}
n=ap_m(x)+bp_m(y)+cp_m(z).
\end{align}
We can obtain a positive definite quadratic form by completing the square, which we denote by $Q'$: 
\begin{align}
\label{qform}
\nonumber
8(m-2)n+(a+b+c)(m-4)^2 &= a\phi_m(x)^2 + b\phi_m(y)^2 + c\phi_m(z)^2\\
:&=Q'(\phi_m(x), \phi_m(y),\phi_m(z)),
\end{align}
where $\phi_m(x):=2(m-2)x-(m-4)$.

We then prove the following propositions on the $p$-adic representations:
\begin{proposition}\label{localobsodd}
For any odd prime $p$, $P_{a,b,c,m}$ represents every integer $p$-adically, if and only if one of the following holds:
\begin{enumerate}
\item $m\equiv 2\pmod{p}$, and $p\nmid\gcd(a,b,c)$;
\item $m\not\equiv 2\pmod{p}, p\nmid a, p\nmid b$ and $p\nmid c$; or
\item $m\not\equiv 2\pmod{p}$, $p$ divides exactly one of $a,b,c$ (without loss of generality we assume $p|c$),
 and if $p^q||c$ with $q$ odd, then
\[
\left(\cfrac{a}{p}\right)=\left(\cfrac{-b}{p}\right)
\]
has to hold as well.
\end{enumerate}
\end{proposition}

\begin{proof}\quad
\begin{enumerate}
\item
Write $m-2=p^qr$, with $p\nmid r$. Wlog assume $p\nmid c$. Expanding equation(\ref{qform}) modulo $p^k$, and subtracting $(a+b+c)(m-4)^2$ from both sides, we obtain
\begin{align*}
8p^qrn &\equiv 4(ax^2+by^2+cz^2)p^{2q}r^2 - 4(m-4)(ax+by+cz)p^qr \pmod{p^k}.\\
\intertext {dividing through $p^q$ gives}
8rn &\equiv 4(ax^2+by^2+cz^2)p^qr^2 - 4(m-4)(ax+by+cz)r \pmod{p^{k-q}}.
\end{align*}
Since $p\nmid 4(m-4)r$, the term $4(m-4)(ax+by+cz)r$ runs through every residue classes modulo $p$. Hence, for some $c_0\in\Z$, the equation
\begin{align*}
8rn &\equiv 4(ax^2+by^2+cz^2)p^qr^2 - 4(m-4)(ax+by+cz)r \pmod{p}
\end{align*}
is satisfied. Then, by replacing $c_0$ with $c_0+c_1p$ for some $c_1\in\Z$, the equation 
\begin{align*}
8rn &\equiv 4(ax^2+by^2+cz^2)p^qr^2 - 4(m-4)(ax+by+cz)r \pmod{p^2}
\end{align*}
can be satisfied. Repeating this until the power of $p$ reaches $k$ implies the statement.

\item \label{qred}
When $m\not\equiv 2\pmod{p}$, we have $Q'\cong \langle a,b,c\rangle$. Computing the Hilbert symbols yields
\begin{align*}
S_pQ' = (a,b)_p(a,c)_p(b,c)_p = 1 = (-1, -abc)_p.
\end{align*}
So $Q'$ is isotropic and hence universal in $\Q_p$.

\item
Wlog assume $c=p^q u$, where $u\in\mathfrak{u}$. Computing the Hilbert symbols yields
\begin{align*}
S_pQ' = (a,b)_p(a,c)_p(b,c)_p = \Big(\frac{a}{p}\Big)^q\Big(\frac{b}{p}\Big)^q,
\end{align*}
while
\begin{align*}
(-1, -abc)_p = \Big(\frac{-1}{p}\Big)^q.
\end{align*}
Hence, $Q'$ is universal in $\Q_p$ if and only if
\begin{align*}
\Big(\frac{a}{p}\Big)^q\Big(\frac{b}{p}\Big)^q = \Big(\frac{-1}{p}\Big)^q,
\end{align*}
which trivially holds for $q$ even, and is equivalent to the stated extra condition when $q$ is odd.
\end{enumerate}
\end{proof}

Meanwhile, for $p=2$, the permissible combinations of $(a,b,c,m)$ are less clearly given, but we can determine whether there is a local obstruction by finite checking.

\begin{proposition}\label{localobs2} For $p=2$, we have the following:
\begin{enumerate}
\item
If $4\nmid m$, then $P_{a,b,c,m}$ has no local $2$-adic obstruction if and only if $2\nmid\gcd(a,b,c)$. 
\item
If $4|m$, then $P_{a,b,c,m}$ has no local $2$-adic obstruction if and only if $(a,b,c)\pmod{8}$ is one of the following tuples (wlog we assume $|a|\leq |b|\leq |c|$):
\begin{align*}
&\pm(1,1,3) & &\pm(1,1,6) & &\pm(1,1,7) & &\pm(1,2,5) & &\pm(1,2,7) & &\pm(1,3,3)\\
&\pm(1,3,5) & &\pm(1,3,7) & &\pm(2,3,3) & &\pm(2,3,5) & &\pm(3,3,5)
\end{align*}
\end{enumerate}
\end{proposition}

\begin{proof}
For odd $m$, the result is trivial, by expanding equation (\ref{qform}) modulo $2^k$.
\quad\\
For $m\equiv 2\pmod{4}$, then $\phi_m(x)\equiv 2\pmod{4}$. Then we can divide both sides (\ref{qform}) by 4:
\begin{align}\label{qformd4}
2(m-2)n+(a+b+c)\Big(\cfrac{m-4}{2}\Big)^2=a\Big(\cfrac{\phi_m(x)}{2}\Big)^2+b\Big(\cfrac{\phi_m(y)}{2}\Big)^2+c\Big(\cfrac{\phi_m(z)}{2}\Big)^2.
\end{align}
Then the result is easy to see by expanding equation (\ref{qformd4}) modulo $2^k$.
\quad\\
If $4|m$, then $4|\phi_m(x)$. Then, we can divide both sides of the equation by 16:
\begin{align*}
\Big(\cfrac{m-2}{2}\Big)n+(a+b+c)\Big(\cfrac{m-4}{4}\Big)^2=a\Big(\cfrac{\phi_m(x)}{4}\Big)^2+b\Big(\cfrac{\phi_m(y)}{4}\Big)^2+c\Big(\cfrac{\phi_m(z)}{4}\Big)^2.
\end{align*}
Notice that $\phi_m(x)/4$ runs through all congruence classes modulo $2^k$. So the problem becomes whether 
\begin{align*}
ax^2+by^2+cz^2\equiv s\pmod{2^k}
\end{align*}
is solvable for every value of $s$ and for all $k\in\N$.\\

This problem can be solved by showing the existence of a solution that can be lifted for some $k\in\N$. By computation, the smallest $k$ that works for all the cases is $k=5$. Checking the permissible combinations yields the list shown in the proposition statement.
\end{proof}

With the two propositions above, given any form $P_{a,b,c,m}$, we can easily check if there is a local obstruction. In particular, it suffices to check for local obstructions for $p=2$ as well as prime divisors of $a,b,c$.

\section{Generating Modular Forms}

The aim of the section is to prove Theorem \ref{aucrit}. Here we assume the form $P_{a,b,c,m}$ has no local obstructions. To start with, we write $N:=2(m-2)$, $\ell_n=8(m-2)n+(a+b+c)(m-4)^2$, and $r(\ell_n)$ be the number of solutions $(x,y,z)$ to equation (\ref{qform}). Consider the generating function:
\begin{align*}
\Theta(\tau)=\sum\limits_{n\geq 0}r(\ell_n)q^\frac{\ell_n}{2N}.
\end{align*}
From  \cite[Proposition 2.1]{Shimura1/2}, the theta function
\begin{align*}
\theta(\tau) = \sum_{x\equiv -(m-4)\pmod{N}} q^\frac{x^2}{2N}
\end{align*}
is a weight $1/2$ modular form on $\Gamma_1(4(m-2))$ with some multipliers.
By the transformation $\theta_a(\tau)=\theta(a\tau)$, we obtain that
\begin{align*}
\theta_a(\tau) = \sum_{x\equiv -(m-4)\pmod{N}} q^\frac{ax^2}{2N}
\end{align*}
is a weight $1/2$ modular form on $\Gamma_1(2aN)$ with some multipliers.
Hence, the product
\begin{align}
\Theta(\tau)=\theta_a(\tau)\theta_b(\tau)\theta_c(\tau)= \sum_{x,y,z\equiv -(m-4)\pmod{N}} q^\frac{ax^2+by^2+cz^2}{2N}
\end{align}
is a weight $3/2$ modular form on $\Gamma_1(2wN)$, where $w:=\lcm(a,b,c)$, with some multipliers.

We follow the approach of Haensch and Kane \cite{HaenschKane}, and decompose $\Theta(\tau)$ into three parts:
\begin{align}
\Theta(\tau)=\mathcal{E}(\tau)+\mathcal{U}(\tau)+f(\tau),
\end{align}
where $\mathcal{E}$ is a linear combination of Eisenstein series, $\mathcal{U}$ is a linear combination of unary theta functions, and $f(\tau)$ is a cusp form that is orthogonal to unary theta functions. Clearly, each part has the same modularity as $\Theta(\tau)$.

\subsection{Coefficients of Eisenstein series}
Let $Q(x,y,z)$ be a ternary quadratic form, and $\mathcal{N}$ be an infinite subset of $\N$. A prime $p$ is called anisotropic if the equation $Q(x,y,z)=0$ has no non-trivial solutions in $\Z_p$, the ring of $p$-adic integers. Kane and Sun \cite{KaneSun} showed that if $\mathcal{N}$ has bounded divisibility at every anisotropic prime (i.e. for all anisotropic primes $p_i$, there exists corresponding indices $k_i$ such that $p_i^{k_i}\nmid n$ for all $n\in \mathcal{N}$), then the $n$-th Fourier coefficient of the Eisenstein series $\mathcal{E}$ corresponding to the quadratic form (as defined by Siegel-Weil formula) is greater than $n^{1/2-\epsilon}$ for sufficiently large $n\in\mathcal{N}$. To establish a lower bound of Fourier coefficients of $\mathcal{E}$, we need the following lemma:

\begin{lemma}\label{anisotropic}
If the form $P_{a,b,c,m}$ has no local obstructions, then the set
\begin{align*}
\mathcal{N}':=\{8(m-2)n+(a+b+c)(m-4)^2 | n\in\N\}
\end{align*}
has bounded divisibility at every anisotropic prime for the quadratic form  $Q'(x,y,z):= ax^2+by^2+cz^2$.
\end{lemma}

We prove Lemma \ref{anisotropic} in a series of lemmas below, where we consider different cases depending on the  divisibility of $abc$ and $m-2$ by the given primes. For any prime $p$ and $k\in\N$, a solution $(x_0,y_0,z_0)\in(\Z/p^k\Z)^3$ to $Q'(x,y,z)=n\pmod{p^k}$ is called \textit{primitive} if $\gcd(x_0,y_0,z_0)$ is relatively prime to $p$ and {\it imprimitive} otherwise. Clearly, if $p$ is anisotropic, then the equation
\[
Q'(x,y,z)\equiv 0\pmod{p^k}
\]
can have no primitive solutions for large enough $k$. Therefore, to show that $p$ is not anisotropic, it suffices to show that a primitive solution exists for every $k\in\N$. Our strategy to prove Lemma \ref{anisotropic} is to prove that for every prime $p$, either
\begin{enumerate}[label = (\roman*)]
\item $Q'$ is isotropic in $\Q_p$;
\item there exists a primitive solution to $Q'(x,y,z)\equiv 0\pmod{p^k}$ for every $k\in\N$; or
\item $\mathcal{N}'$ has bounded divisibility at $p$.
\end{enumerate}

We begin by considering the case when $p\nmid abc$.

\begin{lemma}\label{seriesstart}
If $p$ is odd and $p\nmid a,b,c$, then $p$ is not anisotropic.
\end{lemma}
\begin{proof}
Invoking (\ref{isocrit}), we obtain that
\begin{align*}
S_pQ' = 1 = (-1, -dQ')_p.
\end{align*}
Hence $Q'$ is isotropic in $\Q_p$, so $p$ is not anisotropic.
\end{proof}

\begin{lemma}
If $p$ is odd, $m\not\equiv 2\pmod{p}$, $p$ divides exactly one of $a,b,c$ (without loss of generality we assume $p|c$),
 and 
\[
\left(\cfrac{a}{p}\right)=\left(\cfrac{-b}{p}\right)
\]
when $p^q||c$ with $q$ odd, then $p$ is not anisotropic.
\end{lemma}
\begin{proof}
Assuming the conditions, we obtain that
\begin{align*}
S_pQ' =  \Big(\frac{a}{p}\Big)^q\Big(\frac{b}{p}\Big)^q = \Big(\frac{-1}{p}\Big)^q = (-1, -dQ').
\end{align*}
Hence $p$ is not anisotropic.
\end{proof}

\begin{lemma}\label{lem:p2modm-2}
If $p$ is odd, $m\equiv 2\pmod{p}$, $p|abc$ and $p\nmid\gcd(a,b,c)$, then $p$ is either not anisotropic or $\mathcal{N}'$ has bounded divisibility at $p$.
\end{lemma}
\begin{proof}
Wlog we assume $p|c$. If $p\nmid a+b$, then it is easy to check that for all $n\in\mathcal{N}'$, we have $p\nmid n$, and hence we have a bounded divisibility at $p$. The other possibility is $a\equiv -b\pmod{p}$. Then, for any $k\in\N$ and $r\in\Z$, we can always find $x,y$ relatively prime to $p$ such that 
\[
ax^2+by^2\equiv rp\pmod{p^k}.
\]
Hence we obtain a primitive solution to $Q'(x,y,z)\equiv 0\pmod{p^k}$.
\end{proof}

\begin{lemma}
If $4\nmid m$, then 2 is not an anisotropic prime, or $\mathcal{N}'$ has bounded divisibility at 2.
\end{lemma}
\begin{proof}
For the case $4\nmid m$, if $8\nmid a+b+c$, then $\mathcal{N}'$ has bounded divisibility at $2$. Meanwhile, if $8|a+b+c$, the assumption that $P_{a,b,c,m}$ has no local obstructions implies that exactly one number among $a,b,c$ is even (wlog we assume $2|c$). Clearly, we have a primitive solution to $Q'(x,y,z)\equiv 0\pmod{8}$, namely $(1,1,1)$. Notice that for odd $a$, the congruence equation
\[
ax^2\equiv 8r+a\pmod{2^k}
\]
has an odd solution for $x$ for all $r\in\Z$ and $k\in\N$. This implies that $Q'(x,y,z)\equiv 0\pmod{2^k}$ has a primitive solution for all $k\in\N$.
\end{proof}

\begin{lemma}\label{seriesend}
If $4|m$, and $P_{a,b,c,m}$ has no local 2-adic obstruction, then 2 is not an anisotropic prime.
\end{lemma}
\begin{proof}
For the case $4|m$, the assumption that $P_{a,b,c,m}$ has no local obstructions implies that $a,b,c$ must be congruent modulo $8$ to one of the tuples listed in Proposition \ref{localobs2}. Similar to the case above, at most one number among $a,b,c$ is even (wlog we assume $2|c$); and by direct checking, for every relevant $a,b,c$, we also have a primitive solution to $Q'(x,y,z)\equiv 0\pmod{8}$. Using the same argument as in the $4\nmid m$ case, we obtain that $Q'(x,y,z)\equiv 0\pmod{2^k}$ has a primitive solution for all $k\in\N$.
\end{proof}

\begin{proof}[Proof of Lemma \ref{anisotropic}]
As Lemma \ref{seriesstart} to \ref{seriesend} covers all forms without local obstructions, Lemma \ref{anisotropic} directly follows.
\end{proof}

Using Lemma \ref{anisotropic}, we can conclude that Fourier coefficients of $\mathcal{E}$ is greater than $\ell^{1/2-\epsilon}$ for all sufficiently large $\ell$ that are supported (i.e. $\ell=\ell_n$ for some $n$, so that the coefficients are not trivially zero). Also, Duke's work \cite{Duke} implies that the $\ell$-th Fourier coefficients of $f$ do not exceed $\ell^{3/7+\epsilon}$ for large enough $\ell$. Hence, if $\mathcal{U}$ is identically zero, then the form $P_{a,b,c,m}$ is almost universal. 
With the lemma, we are ready to prove Theorem \ref{aucrit}:
\begin{proof}[Proof of Theorem \ref{aucrit}]
We know that $\mathcal{U}$ is a linear combination of unary theta functions, which are of this form:
\begin{align}\label{utfdef}
\vartheta_{h,t,N}(\tau):=\sum\limits_{r\equiv h\pmod{\frac{2N}{t}}} rq^\frac{tr^2}{2N},
\end{align}
where $h$ may be chosen modulo $2N/t$, and $t$ is a squarefree divisor of $2N$. 

By splitting $\vartheta_{h,t,N}(\tau)$ into two parts:
\begin{align*}
\vartheta_{h,t,N}(\tau) = \vartheta_{h,t,N}(\tau)^+ + \vartheta_{h,t,N}(\tau)^- :=\sum\limits_{\substack{r\equiv h\pmod{\frac{2N}{t}}\\r>0}} rq^\frac{tr^2}{2N} + \sum\limits_{\substack{r\equiv h\pmod{\frac{2N}{t}}\\r<0}} rq^\frac{tr^2}{2N,}
\end{align*}
we observe that the negative Fourier coefficients of $\vartheta_{h,t,N}(\tau)^-$ have a growth rate greater than $\ell^{1/2-\epsilon}$, which cannot be compensated by $f$. Also, we observe that the Fourier coefficients of  $\Theta(\tau)$ are non-negative, so the powers at which $\vartheta_{h,t,N}(\tau)^-$ has nonzero coefficients must coincide with the powers at which $\mathcal{E}$ has nonzero (positive) coefficients. Hence, such unary theta functions must satisfy
\begin{align*}
\ell_n=tr^2\equiv 0\pmod{t},
\end{align*}
which implies
\begin{align*}
\ell_n=8(m-2)n+(a+b+c)(m-4)^2\equiv 0\pmod{t}.
\end{align*}
As $t$ is squarefree, any square factors in $\ell_n$ must come from $r$, and can be divided through. So we can write $\ell_n=s\ell'_n$, where $s$ is the square part of $\ell_n$, and then we have $t|\ell'_n$. Now, as $t$ divides both $\ell'_n$ and $2N=4(m-2)$, and is squarefree, so $t$ satisfies:
\begin{align*}
t\big|\gcd\big(4(m-2),\ell'_n\big).
\end{align*}
Hence, we have a finite (and usually small) set $T$ of possible values of $t$. For each $t$, we check if there exists $n,r\in\N$ such that $tr^2=\ell_n$. In particular, we check where the following congruence equation
\begin{align*}
tr^2\equiv(a+b+c)(m-4)^2\pmod{8(m-2)}
\end{align*}
is solvable. If the congruence equation is not solvable for all possible values of $t$, we can conclude that $\mathcal{U}$ is zero, and hence the form $P_{a,b,c,m}$ is almost universal.
\end{proof}

\subsection{A worked example}\quad\label{Pformexample}

We consider the form $\mathbf{P}_m:=p_m(x)+p_m(y)+3p_m(z)$. The only prime factor among the weights is 3, and we have
\[
\Big(\cfrac{1}{3}\Big)\neq\Big(\cfrac{-1}{3}\Big).
\]
By Theorem \ref{localobsodd}, there is a local obstruction modulo $9$ for $m\not\equiv 2\pmod{3}$. Meanwhile, by Theorem \ref{localobs2}, we conclude that there are no local obstruction modulo $2^k$ for any value of $m$.

Hence, the form $\mathbf{P}_m$ has no local obstruction if and only if $m\equiv 2\pmod{3}$. In this case, we have
\[
\ell_n=8(m-2)n+5(m-4)^2.
\]
So, the value of $t$ is restricted by
\[
t\Big|\gcd\Big(4(m-2),\cfrac{8(m-2)n+5(m-4)^2}{s}\Big).
\]
By simple calculation, we have $T=\{1,5\}$ for $m\equiv 2\pmod{5}$ and $T=\{1\}$ otherwise.

For $t=1$, the relevant congruence equation is
\[
r^2\equiv 5(m-4)^2\pmod{8(m-2)}.
\]
If $4\nmid m$, then this is clearly not solvable, by consider modulo 8. Meanwhile, if $4|m$, it is solvable if and only if $(m-2)/2$ is a product of primes that are congruent 1 or 4 modulo 5 (so that $(5/p_i)=1$). However, as we has already assumed $m\equiv 2\pmod{3}$, this implies $(m-2)/2$ is divisible by 3, and hence it is not solvable either.

For $t=5$, the relevant congruence equation is
\[
5r^2\equiv 5(m-4)^2\pmod{8(m-2)},
\]
which is solvable, hence we cannot conclude anything for $m\equiv 2\pmod{5}$ with this method.

Thus, we conclude that for $m\equiv 2\pmod{3}$ and $m\not\equiv 2\pmod{5}$, the form $\mathbf{P}_m$ is almost universal.

\section{Missing Representations of the Spinor Genus}

In this section, we propose an algorithm which can be used to show that some specific forms $P_{a,b,c,m}$ are not almost universal. This section closely follows the approach introduced in Section 5 of \cite{HaenschKane}, and serves as its generalisation.

\subsection{A special case}

For an individual case, we can verify the conjecture \ref{sswconj}: 
\[
\Big(\sum\limits_{M+\mu\in\cls/\spn(L+\nu)}\omega_{M+\mu}^{-1}\Big)^{-1} \sum\limits_{M+\mu\in\cls/\spn(L+\nu)}\frac{\Theta_{M+\mu}}{\omega_{M+\mu}} = \mathcal{E}_{\spn(L+\nu)}+\mathcal{U}_{\spn(L+\nu)}
\]
and find out the unary theta function in the conjecture. If the Fourier coefficients of the unary theta functions exactly cancels that of the Eisenstein for infinitely many coefficients, then those Fourier coefficients of the weighted sum are zero. Since the Fourier coefficients of the theta functions are non-negative, we can conclude that those Fourier coefficients of our generating functions are zero, proving that it is not almost universal. Here we work on the form $\mathbf{P}_m$ as described in Section \ref{Pformexample}. Specifically, we prove Proposition \ref{notauintro}, in a more concrete form:
\begin{proposition}\label{notau}
The form $\mathbf{P}_{17}=p_{17}(x)+p_{17}(y)+3p_{17}(z)$ is not almost universal. In particular, $\mathbf{P}_{17}$ does not represent $n$ whenever
\begin{align*}
120n+845 = 5\ell^2,
\end{align*}
where $\ell$ is a prime that is congruent to 1 or $19\pmod{30}$.
\end{proposition}

An alternative way to write the quadratic form (\ref{qform}) is to make use of a shifted lattice. By substituting $(a,b,c,m)=(1,1,3,17)$, we obtain:
\begin{align*}
L&:=\lan 30^2, 30^2, 3\cdot 30^2 \ran, & \nu&:=-\frac{13}{30}(e_1+e_2+e_3),
\end{align*}
where $\{e_1,e_2,e_3\}$ is the standard basis. For simplicity, we shall write the cosets by their coordinates with respect to this standard basis. Then, we can rewrite equation (\ref{qform}) in terms of a quadratic form
\begin{align}\label{latticeqform}
120n+845=Q_L(\omega+\nu),
\end{align}
where $Q_L \lan 30^2, 30^2, 3\cdot 30^2 \ran$, $\omega\in\Z^3$.
We now proof a small lemma about the genus of $L+\nu$:
\begin{lemma}
Every lattice coset in the genus of $L+\nu$ can be written in the form $L+\mu$ for some vector $\mu$ such that $30\mu\in L$.
\end{lemma}
\begin{proof}
By definition, if $M+\mu \in \gen(L+\nu)$, then for every $p$ there exists $X_p\in O(V_p)$ such that $X_p(M+\mu)_p = (L+\nu)_p$. 

For $p\neq 2,3,5$, we have
\begin{align*}
M_p+\mu = (M+\mu)_p \cong (L+\nu)_p = L_p.
\end{align*}
This implies $M_p+\mu$ is actually a lattice, that is, $\mu\in M_p$. Hence $X_p(M_p)=L_p$.

Meanwhile, for $p=2,3$ or $5$, since $X_p(M+\mu)_p = (L+\nu)_p$, we know that $X_p(\mu)=\alpha+\nu$ for some $\alpha\in L_p$. Note that $p\nu\in L_p$, and hence $L_p+p\nu=L_p$. Then, we have
\begin{align*}
X_p(M_p+p\mu) = X_p(M_p + \mu) + (p-1)X_p(\mu) = (L_p+\nu)+ (p-1)(\alpha+\nu) = L_p + p\nu = L_p.
\end{align*}
This gives $p\mu\in M_p$, and hence $X_p(M_p+p\mu)=X_p(M_p)=L_p$.

As $L$ is the only class in its genus, $X_p(M_p)=L_p$ for all $p$ implies $M \cong L$. Also, as $M_p+30\mu = M_p$ for every $p$, we have $30\mu\in M\cong L$. Hence the result.
\end{proof}

Via direct computation, we obtain a list of the possible values of $\mu$. There are in total 26 classes in the genus of $L+\nu$. To proceed, we rely on the following result:
\begin{lemma}
The genus of $L+\nu$ splits into two spinor genera, each containing 13 classes.
\end{lemma}
\begin{proof}
This result can be shown by direct computation of the spinor norm of the element in $O(V)O'_{\mathbb{A}}(V)$ that maps $L+\nu$ to $L+\mu$. In particular, we find a map $X = (X_2, X_3, \cdots, X_p, \cdots)\in O(V)O'_{\mathbb{A}}(V)$ such that $X_p(L_p+\nu) = L_p+\mu$ for all $p$. 

For primes other than $2,3,5$, we have 
\begin{align*}
L_p+\mu = L_p = L_p+\nu.
\end{align*}
So it suffices to find $X_2, X_3, X_5$. For example, for $\mu = \frac{1}{30}(7e_1+7e_2+13e_3)$, the map taking $L+\nu$ to $L+\mu$ is given by $X = (X_2, X_3, \cdots, X_p, \cdots) = (\id, \id, \tau_{e_1}\circ\tau_{e_2}, \cdots, \id, \cdots)$. As $Q(e_1)/\mathbb{Q}^2 = Q(e_2)/\mathbb{Q}^2 = 1$, it is clear that $X\in O'_{\mathbb{A}}(V)\subseteq O(V)O'_{\mathbb{A}}(V)$. The other classes are then classified similarly into the two spinor genera, and the explicit computation is omitted here.
\end{proof}
With these lemmas, we can establish the following result:
\begin{proposition}\label{ssw113prop}
We have
\begin{align}\label{ssw113}
\Theta_{\spn(L+\nu)}(\tau)=\mathcal{E}_{\gen(L+\nu)}(\tau)-\frac{\vartheta'(\tau)}{36},
\end{align}
where $\mathcal{E}_{\gen(L+\nu)}$ is the Eisenstein series obtained by Shimura's generalised Siegel-Weil formula (equation (\ref{ShimuraSW})), and
\[
\vartheta'(\tau)=\sum\limits_{n\equiv 1\pmod{6}} nq^{5n^2}
\]
is a unary theta function.
\end{proposition}
\begin{proof}
The valence formula states that a weight $k$ modular form on a congruence subgroup $\Gamma$ can be uniquely determined by the first
\begin{align*}
\frac{k}{12}[\SL_2(\Z):\Gamma]
\end{align*}
coefficients of its Fourier expansion, where $[\SL_2(\Z):\Gamma]$ denotes the index of $\Gamma$ in $\SL_2(\Z)$. For our case, the forms in equation (\ref{ssw113}) are all modular on the congruence subgroup $\Gamma_1(4\cdot 3\cdot 30^2)$; and the index $[\SL_2(\Z):\Gamma_1(N)]$ is given by the following formula \cite[p.~2]{Ono}:
\begin{align}
[\SL_2(\Z):\Gamma_1(N)] = N^2\prod_{p|N}(1-p^{-2}).
\end{align}
With this, we can find the number of terms we need, by a simple substitution:
\[
\frac{1}{8}(10800)\Big(\frac{3}{4}\Big)\Big(\frac{8}{9}\Big)\Big(\frac{24}{25}\Big)=9331200.
\]

So we can confirm that the two sides are indeed equal, by checking the first 9331200 terms of the Fourier expansions, which can be easily done with a computer.
\end{proof}

The Fourier coefficients of the unary theta function in equation (\ref{ssw113}) is explicit, while that of the Eisenstein series is less so. In order to show that their coefficients indeed cancel out each other for infinitely many terms, we need the following lemma:
\begin{lemma}\label{Eisencoeff}
Let $\ell$ be a prime that is congruent to either 1 or 19 modulo 30. Then, the $5\ell^2$-th coefficient of $\mathcal{E}_{\gen(L+\nu)}$ is $\ell/36$.
\end{lemma}
\begin{proof}
This proof adopts the approach introduced in \cite{HaenschKane}. We firstly introduce the sieve operator $S_{N,\ell}$, which acts on Fourier series of the form
\[
f(\tau)=\sum_{n\geq 0} a_nq^n
\]
by the transformation
\[
f|S_{N,\ell}(\tau)=\sum_{\substack{n\geq 0\\ n\equiv \ell\pmod{N}}} a_nq^n.
\]
Now consider the number of representations of $120n+5$ by the quadratic form 
\begin{align}\label{modulorep}
Q(x,y,z) := x^2+y^2+3z^2.
\end{align}
We can classify these representations into two categories:
\begin{enumerate}[label=(\alph*)]
\item \label{oui5} $5|x, 5|y, 5|z$; and
\item \label{non5} everything else.
\end{enumerate}
By equation (\ref{ShimuraSW}), the coefficients of $\mathcal{E}_{\gen(L+\nu)}$ is the sum of the coefficients of shifted theta series in the genus of $L+\nu$, which counts the number of category \ref{non5} representations of the quadratic form (\ref{modulorep}), up to multiplication of a scalar. 

Here we define $\hat{\Theta}:=\Theta^2(\tau)\Theta(3\tau)$, where $\Theta(\tau):=\sum_{n\in\Z}q^{n^2}$ is the Jacobi theta function. Referring back to Shimura's generalised Siegel-Weil formula (equation (\ref{ShimuraSW})), we easily obtain the following equation:
\begin{align}\label{sieveform}
\hat{\Theta}(\tau)|S_{120,5} = 576\mathcal{E}_{\gen(L+\nu)} + \hat{\Theta}(25\tau)|S_{120,5}.
\end{align}
In the equation above, the term $576\mathcal{E}_{\gen(L+\nu)}$ counts the number of category \ref{non5} representations, while the term $\hat{\Theta}(25\tau)|S_{120,5}$ counts the number of category \ref{oui5} representations. Clearly, if $k\in\N$ is not divisible by 25, the $k$-th coefficient of $\hat{\Theta}(25\tau)|S_{120,5}$ is zero. In particular, the $5\ell^2$-th coefficient of $\hat{\Theta}(25\tau)|S_{120,5}$ is zero. Hence, the $5\ell^2$-th coefficient of $\mathcal{E}_{\gen(L+\nu)}$ is concisely determined by the number of representations of
\begin{align}\label{5ellellrep}
x^2+y^2+3z^2=5\ell^2.
\end{align}
By \cite[Theorem 86]{Jones}, the number of representations of the quadratic form (\ref{5ellellrep}) is given by $16H(15\ell^2)$, where $H$ is the Hurwitz class number. By the class number formula \cite[Corollary 7.28]{Cox}, we have
\[
H(15\ell^2)=\cfrac{h(-15)}{2}+\cfrac{h(-15\ell^2)}{2}=1+\ell-\Big(\cfrac{-15}{\ell}\Big),
\]
where $(-15/\ell)$ is the Kronecker symbol. In particular, for $\ell\equiv 1,19\pmod{30}$, we have $(-15/\ell)=1$, and hence we have $H(15\ell^2)=\ell$.

Here we write $\mathcal{E}_{5\ell^2}$ as the $5\ell^2$-th coefficient of $\mathcal{E}_{\gen(L+\nu)}$. Then, by equation (\ref{sieveform}), we have
\[
16\ell = 16H(15\ell^2)=576\mathcal{E}_{5\ell^2},
\]
which gives
\[
\mathcal{E}_{5\ell^2}=\frac{16\ell}{576}=\frac{\ell}{36}.
\]
\end{proof}
With Lemma \ref{Eisencoeff}, we are finally able to prove Proposition \ref{notau}:
\begin{proof}[Proof of Proposition \ref{notau}]
Consider the $5\ell^2$-th coefficients of equation (\ref{ssw113}). 

Since $\ell\equiv 1\pmod{6}$, the $5\ell^2$-th coefficient of $\vartheta'(\tau)$ is $\ell$; meanwhile, Lemma \ref{Eisencoeff}, states that the $5\ell^2$-th coefficient of $\mathcal{E}_{\gen(L+\nu)}$ is $\ell/36$. Therefore, the $5\ell^2$-th coefficient of $\Theta_{\spn(L+\nu)}$ is
\[
\frac{\ell}{36}-\frac{\ell}{36}=0.
\]
As $\Theta_{\spn(L+\nu)}$ is a weighted sum of some theta series $\Theta_{L+\mu}$, whose coefficients are numbers of representations, which are non-negative, we can conclude that the $5\ell^2$-th coefficient of every term in this sum is zero. In particular, the $5\ell^2$-th term of $\Theta_{L+\nu}$ is zero, which implies that there exists no $\omega\in\Z^3$ that satisfies
\begin{align*}
Q(\omega+\nu)=5\ell^2.
\end{align*}
By substituting the parameters of $\mathbf{P}_{17}$ into equation (\ref{latticeqform}), we obtain the criterion for which  $n$ is not represented by $\mathbf{P}_{17}$ as stated.
\end{proof}

\subsection{A slight generalisation}
With Proposition \ref{ssw113prop} and Lemma \ref{Eisencoeff}, we can actually prove more than a single form being not almost universal. Building off from them, we obtain a proof for Theorem \ref{notaugenintro}, in the form of the following proposition:
\begin{proposition}\label{notaugen}
For any non-negative integer $r\not\equiv 2\pmod{5}$, the form $\mathbf{P}_{30r+17}:=p_{30r+17}(x)+p_{30r+17}(y)+3p_{30r+17}(z)$ is not almost universal. In particular, $\mathbf{P}_{30r+17}$ does not represent $n$ whenever
\begin{align*}
8(30r+15)n + 5(30r+13)^2 = 5\ell^2,
\end{align*}
where $\ell$ is a prime that is congruent to 1 or $19\pmod{30}$.
\end{proposition}
\begin{proof}
We substitute $m=30r+17$ into equation (\ref{qform}):
\begin{align}
\nonumber
&\quad\; (240r+120)n+5(30r+13)^2\\
&= (60rx+30x-30r-13)^2+(60ry+30y-30r-13)^2+3(60rz+30z-30r-13)^2.
\end{align}
Notice that the right hand side is also a quadratic form on the lattice coset $L+\nu$. Hence, by Proposition \ref{ssw113prop} and Lemma \ref{Eisencoeff}, it is clear that $n$ is not represented by $L+\nu$ whenever 
\begin{align*}
(240r+120)n+5(30r+13)^2=5\ell^2,
\end{align*}
where $\ell$ is a prime that is congruent to 1 or $19\pmod{30}$.

It remains to prove that there are infinitely many $n$ that can be written in this form. This is equivalent to the congruence equation
\begin{align*}
\ell^2-(30r+13)^2\equiv 0\pmod{48r+24}
\end{align*}
being solvable.
By the Chinese remainder theorem, it is easy to check that the equation is consistent whenever $48r+24$ is not divisible by $5$, that is, $r\not\equiv 2\pmod{5}$. This finishes the proof.
\end{proof}



\begin{thebibliography}{16}

\bibitem{Cauchy} A. Cauchy, \textit{D\'emonstration du th\'eor\`em g\'en\'eral de Fermat sur les nombres polygones}, M\'em. Sci. Math. Phys. Inst. France \textbf{14} (1813-1815), 177-220; Oeuvres compl\`etes \textbf{VI} (1905), 320-353.

\bibitem{ChanHaensch} W. K. Chan and A. Haensch, {\em Almost universal ternary sums of squares and triangular numbers}, Developments in Mathematics {\sc 31}, Springer-Verlag, New York, 2013, pp. 51-- 62.
\bibitem{ChanOh} W.K. Chan and B.-K. Oh, {\em Almost universal ternary sums of triangular numbers}, Proc. Amer. Math. Soc. \textbf{137} (2009), 3553--3562.

\bibitem{HaenschKane} A. Haensch and B. Kane, \textit{Almost universal ternary sums of polygonal numbers}, \textbf{More Info?}

\bibitem{KaneSun} B. Kane and Z.W. Sun, \textit{On almost universal mixed sums of squares and triangular numbers}, Trans. Amer. Math. Soc. \textbf{362} (2010), 6425-6455.

\bibitem{Shimura1/2} G. Shimura, \textit{On modular forms of half-integral weight}, Ann. Math. \textbf{97} (1973), 440-481.

\bibitem{ShimuraLattCo} G. Shimura, \textit{Inhomogeneous quadratic forms and triangular numbers}, Amer. J. Math. \textbf{126} (2004), 191–214.

\bibitem{SiegelEis} C. Siegel, \textit{\"Uber die analytische Theorie der quadratischen Formen}, Ann. Math. \textbf{36} (1935), 527-606.

\bibitem{Duke} W. Duke, \textit{Hyperbolic distribution problems and half-integral weight Maass forms}, Invent. Math. \textbf{92} (1988), 73-90.

\bibitem{Ono} K. Ono, \textit{The web of modularity: Arithmetic of the coefficients of modular forms and q-series}, Amer. Math. Soc,
(2004).

\bibitem{Jones} B. Jones, \textit{The arithmetic theory of quadratic forms, Carcus Monograph series}, \textbf{10}, Math. Assoc. Amer., Buffalo,
NY, (1950).

\bibitem{Cox} D. Cox, \textit{Primes of the form $x^2 + ny^2$}, Wiley, New York, (1989).

\bibitem{Omeara} O.T. O’Meara, \textit{Introduction to quadratic forms}, Springer-Verlag, New York, (1963).

\bibitem{Gerstein} L. J. Gerstein \textit{Basic quadratic forms}, Amer. Math. Soc., Providence, (2008).

\bibitem{Xu} F. Xu, \textit{Strong approximation for certain quadric fibrations with compact fibers}, Adv. Math. \textbf{281} (2015), 279–295.

\bibitem{SPillot1} R. Schulze-Pillot, \textit{Darstellungsma\ss e von Spinorgeschlectern tern\"arer quadratischer Formen}, J. Reine Angew. Math. \textbf{352} (1984), 114–132.

\bibitem{SPillot2} R. Schulze-Pillot, \textit{Thetareihen positiv definiter quadratischer Formen}, Invent. Math. \textbf{75} (1984), 283–299.

\end{thebibliography}
\end{document}